\documentclass{article}
\usepackage{graphicx}
\usepackage{pstricks}
\usepackage{hyperref}
\usepackage{subfigure}
\usepackage{amssymb,amsthm}
\newcommand{\C}[1]{\ensuremath{\mathcal{#1}}}

\newtheorem{theorem}{Theorem}[section]
\newtheorem{prop}{Proposition}[section]

\newtheorem{defn}{Definition}[section]

\begin{document}

\title{Eigenmodes of a Laplacian on some Laakso Spaces}
\author{Kevin Romeo and Benjamin Steinhurst}
\maketitle
\begin{abstract}
We analyze the spectrum of a self-adjoint operator on a Laakso space using the projective limit construction originally given by Barlow and Evans. We will use the hierarchical cell structure induced by the choice of approximating quantum graphs to calculate the spectrum with multiplicities. We also extend the method for using the hierarchical cell structure to more general projective limits beyond Laakso spaces.

MCS: 34L40 (primary); 34L16; 54B30

\end{abstract}

\section{Introduction}\label{sec:intro}
Laakso's spaces were first introduced in \cite{Laakso2000} to give examples of spaces that have nice analytic properties of any arbitrary dimension greater than one. Among these properties is having a sufficient supply of rectifiable curves connecting pairs of points to not give positive capacity to single points. In \cite{BarlowEvans2004}, Barlow and Evans present a construction of a projective limit space with sufficient conditions that Markov processes can be constructed on the limit space and mention that their construction should be able to create Laakso's spaces as well. Then in \cite{Steinhurst2009} that assertion was proved in detail as well as the construction of a Laplacian operator constructed from both the minimal generalized upper gradients provided in \cite{Laakso2000} and from a Markov process constructed as in \cite{BarlowEvans2004}. In this paper we work in terms of Barlow and Evans construction because their construction as the projective limit of quantum graphs (Sections \ref{sec:quantum} and \ref{sec:proj}) makes describing the action of the Laplacian simpler (Section \ref{sec:lap}).

Working on the approximating graphs allows us to numerically compute eigenvalues and eigenfunctions of the Laplacian on each graph and from the projective system of graphs (Section \ref{sec:proj}) an eigenfunction on one approximating graph can be extended to a corresponding eigenfunction on a more detailed approximating graph so that the spectra of the approximating Laplacians are building up towards the spectrum of the Laplacian on the limit space. This is proved in Section \ref{sec:spectrum}. An account of how the numerical calculations are performed is given in Sections \ref{sec:desc}, \ref{sec:lap}. We begin with a few words on quantum graphs in Section \ref{sec:quantum} and the spectrum of limit 

\section{Quantum Graphs}\label{sec:quantum}
A metric graph is a graph which not only has a metric giving the distance between vertices but can also give distances between points along the edges. Such a graph consist of a set of vertices and a set of edges connecting vertices. A pair of vertices can have more than one edge connecting them, and any vertex not connected to an edge can be ignored. Each edge is given a length, and the distance between any two points is the infimum of the lengths of the paths connecting the two points. Such graphs are best visualized as a series of wires connected together at the vertices where interesting phenomena can occur not just at the vertices but along the ``wires'' as well. If a Hamiltonian operator is defined on the metric graph as well it can then be called a \emph{quantum graph}, the simplest Hamiltonian is the standard Laplacian. A good survey of quantum graphs, their properties, and the properties of the Hamiltonians on them is \cite{Kuchment2004} followed up in \cite{Kuchment2005}.

The Hamiltonian need not be an exotic operator, since a quantum graph is primarily a collection of intervals operators can be built out of one dimensional operators and have their action remain intuitive. A simple operator, and the one used in this paper, is the negative second derivative, 
$$Af(\cdot) = -\frac{d^2f}{dx^2}(\cdot),$$ 
The domain of this operator are those functions on the graph that are twice differentiable on each line segment and such that the directional first derivatives at a vertex sum to zero. This vertex condition is known as (Neumann-)Kerchoff matching conditions and force the boundary terms to vanish when using integration by parts to check the self-adjointness of the operator. If the degree of the vertex is one then it is a boundary vertex and the matching condition reduces to a Neumann boundary condition. We will consider this Laplacian on a series of quantum graph approximations to Laakso spaces where the Laplacians also approximate the Laplacian on the limit space.

Hamiltonians are self-adjoint operators not merely symmetric ones, see \cite{AkkermansEtAll2000} for examples in the physics literature. This means that care has to be taken in identifying their domains. For instance, second differentiation on the unit interval is a common Hamiltonian but the domain has to be restricted by boundary conditions, classical ones being Dirichlet and Neumann. However on a quantum graph there are vertices with many line segments meeting. At these points there are more choices for boundary conditions, zero (Dirichlet), zero derivative (Neumann), one sided derivatives summing to zero (Kirchoff), and so on \cite{Kuchment2004,Kuchment2005}. In Section \ref{sec:lap} we discuss which vertex conditions are most appropriate to these fractals. 

\section{Projective Limits}\label{sec:proj}
The Laakso spaces that we will be analyzing Laplacians on are constructed as projective limits. The original construction was given in \cite{BarlowEvans2004}, shown to produce Laakso spaces in \cite{Steinhurst2009}, and background for projective limits in the category of topological spaces is given in \cite{HockingYoung1988} and \cite{Bourbaki2004}. We start with the definitions of a projective system and the projective limit space and then move on to giving the spectrum of an operator on the projective limit space.

\begin{defn}\label{def:projsys}
A \emph{projective system of topological spaces}, $(F_i,\phi_{i+1,i})_{i=0}^{\infty}$, is a family of spaces and surjective maps where the $F_i$ are topological spaces and $\phi_{i+1,i}:F_{i+1} \rightarrow F_i$ are continuous.
\end{defn}

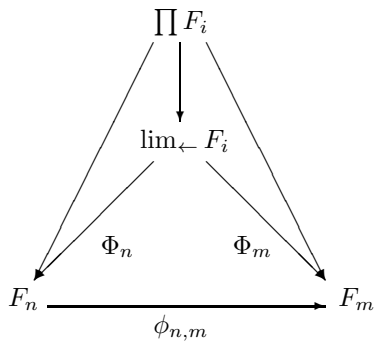
\begin{figure}[t]\centering
\begin{picture}(200,120)
\put(90,115){$\prod F_i$}
\put(85,70){$\lim_{\leftarrow} F_i$}
\put(100,110){\vector(0,-1){30}}
\put(90,110){\vector(-1,-2){45}}
\put(110,110){\vector(1,-2){45}}
\put(90,65){\vector(-1,-1){45}}
\put(110,65){\vector(1,-1){45}}
\put(35,10){$F_n$}
\put(160,10){$F_m$}
\put(50,10){\vector(1,0){105}}
\put(90,00){$\phi_{n,m}$}
\put(70,30){$\Phi_n$}
\put(120,30){$\Phi_m$}
\end{picture}
\caption{Summary of the Projective System, $n > m \ge 0$}
\label{ProjSystem}
\end{figure}

\begin{defn}\label{def:projlim}
The \emph{projective limit of a projective system}, $\lim_{\leftarrow} F_i \subset \prod_{i=0}^{\infty} F_i$ is a topological space. An element $\{x_i\} \in \lim_{\leftarrow}F_i$ if $x_i \in F_i$ and $\phi_{i,i-1}x_i = x_{i-1}$ for all $i \ge 0$ and the maps $\Phi_j:\lim_{\leftarrow}F_i \rightarrow F_j$ are all continuous surjections.
\end{defn}

Projective limits possess a version of the ``universal property.'' This means that if another space $L$ was another topological space satisfying the diagram in Figure \ref{ProjSystem} then all of the maps ``factor'' through $\lim_{\leftarrow}F_i$ inducing a continuous surjection from $L$ to $\lim_{\leftarrow}F_i$. This gives a sense in which the projective limit space is minimal because it is the smallest space which has all of the $\Phi_j$ surjective. The diagram showing this is Figure \ref{UnivProp}, where $L$ is the other candidate for a limit space and $\tilde{\Phi}_j$ the set of surjections belonging to $L$. By the universal property of the projective limit space we have that $\tilde{\Phi}_j = \Phi_j \circ \eta$ for all $J \ge0$ and some $\eta: L \rightarrow \lim_{\leftarrow}F_i$ a continuous surjection.

\begin{figure}[t]\centering
\begin{picture}(200,150)
\put(150,140){$\prod F_i$}
\put(150,80){$\lim_{\leftarrow} F_i$}
\put(30,80){$L$}
\put(160,130){\vector(0,-2){40}}
\put(45,83){\vector(1,0){100}}
\put(95,87){$\eta$}
\put(30,10){$F_n$}
\put(155,10){$F_m$}
\put(35,75){\vector(0,-1){50}}
\put(35,75){\vector(2,-1){110}}
\put(20,45){$\tilde{\Phi}_n$}
\put(65,60){$\tilde{\Phi}_m$}
\put(45,13){\vector(1,0){100}}
\put(160,75){\vector(0,-1){50}}
\put(160,75){\vector(-2,-1){110}}
\put(110,60){$\Phi_n$}
\put(170,45){$\Phi_m$}
\put(90,00){$\phi_{n,m}$}
\end{picture}
\caption{Use of the Universal Property}
\label{UnivProp}
\end{figure}
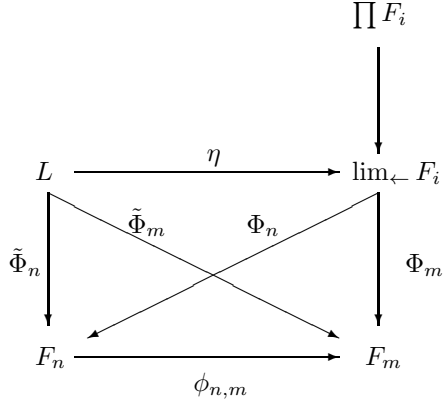

We need not only a topology on the limit space, $\lim_{\leftarrow} F_i$, but also a measure. In \cite{Bourbaki2004} sufficient conditions for the existence of a measure on the limit space are given. They are:
\begin{enumerate}
	\item For $A \subset F_i$ $\mu_i$-measurable that $\int_{F_i} \mathbb{I}_{A}\ d\mu_i = \int_{F_{i+1}} \mathbb{I}_{A} \circ \phi_{i+1,i}\ d\mu_{i+1}$.
	\item The total masses of the measures $\mu_i$ are bounded.
\end{enumerate}
Given these two conditions there exists a limit measure $\mu_{\infty}$ with the property $$\int_{F_i} \mathbb{I}_{A}\ d\mu_i = \int_{F_{i+1}} \mathbb{I}_{A} \circ \Phi_i\ d\mu_{\infty}.$$ This condition forces the sequence of measures $\mu_i$ to be mutually compatible enough that the sequence of measures $\Phi_i^{*}\mu_i(A) := \int \mathbb{I}_A \circ \Phi_i\ d\mu_{\infty}$, with the integral taken over $\lim_{\leftarrow}F_i$, converge to the measure $\mu_{\infty}$ weakly. 

Fix a projective system of quantum graphs, $\{F_n\}_{n=o}^{\infty}$ with measures, $\mu_n$. Let $\Phi_n:\lim_{\leftarrow} F_i \rightarrow F_n$ where $\lim_{\leftarrow}F_i$ is projective limit of $F_n$. It is possible to pull back functions on $F_n$ to functions on $\lim_{\leftarrow}F_i$ by pre-composing with $\Phi_n$. I.e. if $f:F_n \rightarrow \mathbb{R}$ then $\Phi_n^{*}f := f \circ \Phi_n : \lim_{\leftarrow}F_i \rightarrow \mathbb{R}$. 
 
\begin{defn}\label{def:A_n}
Let $A_n$ be a self-adjoint operator on $F_n$ with domain $Dom(A_n)$ consisting of functions on $F_n$ that are twice differentiable on each line segment and and also have the property that at each vertex the one-sided first derivatives along each edge incident to the vertex sum to zero. These are the Kirchoff matching conditions.
\end{defn}

\begin{defn}\label{def:incfam}
A family of self-adjoint operators, $(A_n,Dom(A_n))$ is \emph{increasing} if $Dom(A_n) \subset Dom(A_{n+1})$ and $A_{n+1}f = A_nf$ for all $f \in Dom(A_n)$ and all $n \ge 0$.
\end{defn}

\begin{prop}
The operators $A_n$ with domains 
$$\C{D}_n=\{ f \circ \Phi_n | f \in Dom(A_n) \}$$ 
Acting by $A_n \tilde{f} = (A_n f) \circ \Phi_n$ for $\tilde{f} \in \C{D}_n$ are an increasing family of self-adjoint operators and the operator $A$ is the minimal self-adjoint extension of all of the $A_n$ and has domain $$Dom(A)  = \overline{\bigcup_{n=0}^{\infty} \C{D}_n}.$$ 
\end{prop}

\begin{proof} The proof of this statement is rather lengthy and uses Dirichlet forms so we refer to Section 7 of \cite{Steinhurst2009} for the proof.
\end{proof}

The self-adjoint operator $A$ is referred to as the \emph{projective limit} of $A_n$. We have 
$$\bigcup_{n=0}^{\infty} \C{D}_n  = \bigcup_{n=0}^{\infty} \Phi^{*}_nDom(A_n) \subset Dom(A)$$
Densely in the graph norm, $\|u\| = \|u\|_{L^{2}} + \|Au\|_{L^{2}}$, because of the projective limit construction which defines $Dom(A)$ as the minimal domain containing $\C{D}_n$ for all $n \ge0$ such that $(A,Dom(A))$ is self-adjoint. It can be checked that the graph norm closure of $\bigcup_{n=0}^{\infty} \C{D}_n$ is that domain. To ease the calculations that follow we orthogonalize the spaces $\C{D}_n$ as follows:
 \begin{eqnarray*}
 	\mathcal{D}'_0 &=& Dom(A_0)\\
	\mathcal{D}'_1 &=& \mathcal{D'}_0^{\perp} \cap \C{D}_1\\
	\mathcal{D}'_n &=& \mathcal{D'}_{n-1}^{\perp} \cap \C{D}_n.
\end{eqnarray*}
Where $\mathcal{D}_{n-1}^{\perp}$ is the orthogonal complement of $\mathcal{D}_{n-1}$ in $\mathcal{H} = L^{2}(L,\mu)$ and $\oplus_{n=0}^{\infty} \mathcal{D'}_n = \bigcup_{n=0}^{\infty} \C{D}_n.$ If one can find the spectrum of the Laplacian, $-A$, when operating on each $\mathcal{D'}_n$ then the spectrum of $-A$ when acting on its entire domain is the union of its spectrum on the $\mathcal{D'}_n$ with multiplicities adding. We have proved the following theorem.

\begin{theorem}\label{thm:orthodom}
Given a projective system of quantum graphs, if $(A,Dom(A))$ is the projective limit of the operators $(A_n,Dom(A_n))$ then 
$$\sigma(A) = \bigcup_{n=0}^{\infty} \sigma(\left. A_n\right|_{\C{D}'_n}).$$
Where multiplicities of the eigenvalues add.
\end{theorem}

\begin{proof} 
Above we saw that $\oplus_{n=0}^{\infty} \mathcal{D'}_n = \bigcup_{n=0}^{\infty} \C{D}_n.$.
If the union, $\bigcup_{n=0}^{\infty} \sigma(\left. A_n\right|_{\C{D}'_n})$, did not account for all of the eigenvalues of $A$ then this would imply that that there was at least a one dimensional subspace of $Dom(A)$ which is orthogonal to every $\C{D}_n$ which is not true because $A$ is defined as the projective limit of the operators, $A_n$. Thus a complete set of eigenfunctions for $Dom(A)$ can be found in $\oplus_{n=0}^{\infty} \C{D'}_n$. A more detailed version of this proof can be found in \cite{Steinhurst2009}.
\end{proof}

If the geometry of a projective limit space is well understood then it is possible to explicitly determine the orthogonalization, $\C{D'}_n$. This is done in Section \ref{sec:spectrum} in the case of Laakso spaces.

\section{Laakso Spaces}\label{sec:desc}
Laakso originally described these spaces in \cite{Laakso2000} with emphasis on showing there was a large supply of rectifiable curves and a tunable dimension. Later Barlow and Evans created a projective limit construction that could be specialized to use quantum graphs that would also produce Laakso spaces. This was verified in \cite{Steinhurst2009}. We consider only a few examples of these spaces to calculate spectra of Laplacians, however we take advantage of two ways of visualizing the process. One of the factors in determining the dimension of Laakso spaces is how many subintervals the new identifications made at each level of approximation make, we consider only the cases when this number, $j$, is fixed from one level to the next and work with $j=2,3,4,5,6,7$.

\begin{figure}[t]\center
\input{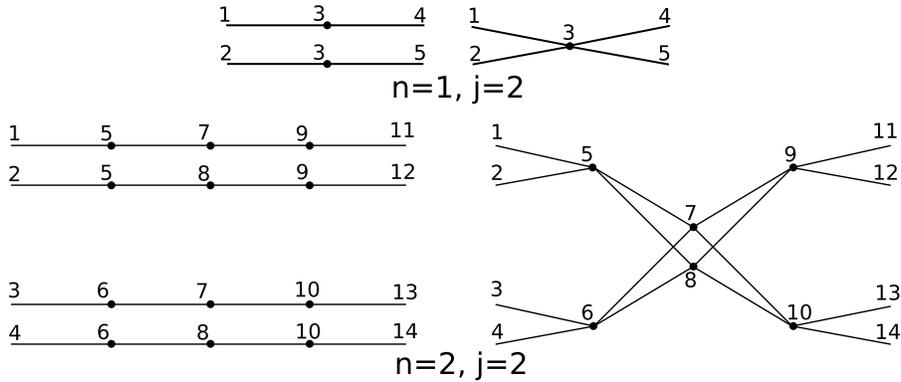}
\caption{Constructions of the simplest case, $j=2$.}\label{repfig}
\end{figure}

\begin{figure}[t]\center
\input{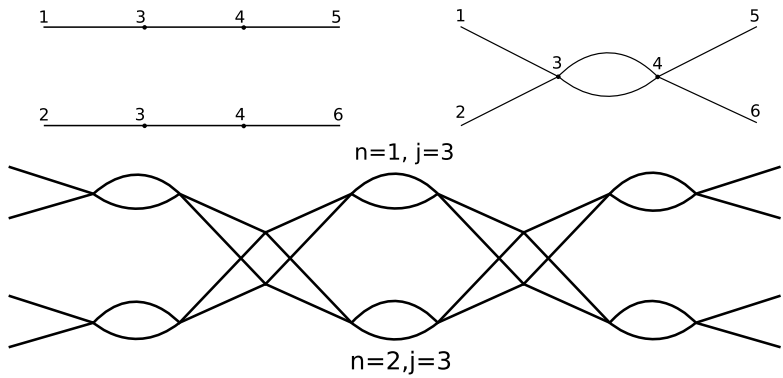}
\psset{xunit=.5pt,yunit=.5pt,runit=.5pt}
\begin{pspicture}(585.93554688,157.03961182)
{
\newrgbcolor{curcolor}{0 0 0}
\pscustom[linewidth=1.95507383,linecolor=curcolor]
{
\newpath
\moveto(0.977537,156.06207482)
\lineto(65.096678,135.55328682)
\curveto(65.096678,135.55328682)(95.224319,110.15839482)(130.194334,135.55328682)
\lineto(195.311522,106.25641182)
\lineto(260.410162,135.55328682)
\curveto(260.410162,135.55328682)(292.202062,110.57470482)(325.506832,135.55328682)
\lineto(390.624022,106.25641182)
\lineto(455.721672,135.55328682)
\curveto(455.721672,135.55328682)(487.514542,106.82791682)(520.819332,135.55328682)
\lineto(584.958012,156.06207482)
}
}
{
\newrgbcolor{curcolor}{0 0 0}
\pscustom[linewidth=1.95507383,linecolor=curcolor]
{
\newpath
\moveto(0.977537,116.99957482)
\lineto(65.096678,135.55328682)
\curveto(101.315631,168.02544572)(130.194334,135.55328682)(130.194334,135.55328682)
\lineto(195.311522,67.19391282)
\lineto(260.409182,135.55328682)
\curveto(295.795512,169.69068318)(325.506832,135.55328682)(325.506832,135.55328682)
\lineto(390.624022,67.19391282)
\lineto(455.721672,135.55328682)
\curveto(489.859082,171.35592442)(520.819332,135.55328682)(520.819332,135.55328682)
\lineto(584.958012,116.99957482)
}
}
{
\newrgbcolor{curcolor}{0 0 0}
\pscustom[linewidth=1.95507383,linecolor=curcolor]
{
\newpath
\moveto(0.977537,58.40582282)
\lineto(65.096678,37.89705282)
\curveto(98.817772,73.69968082)(130.194334,37.89705282)(130.194334,37.89705282)
\lineto(195.311522,106.25641182)
\lineto(260.409182,37.89705282)
\curveto(260.409182,37.89705282)(290.953112,9.58799282)(325.506832,37.89705282)
\lineto(390.624022,106.25641182)
\lineto(455.721672,37.89705282)
\curveto(455.721672,37.89705282)(487.098242,10.00429282)(520.819332,37.89705282)
\lineto(584.958012,58.40582282)
}
}
{
\newrgbcolor{curcolor}{0 0 0}
\pscustom[linewidth=1.95507383,linecolor=curcolor]
{
\newpath
\moveto(0.977537,19.34332282)
\lineto(65.096678,37.89705282)
\curveto(65.096678,37.89705282)(95.224319,10.42060282)(130.194334,37.89705282)
\lineto(195.311522,67.19391282)
\lineto(260.409182,37.89705282)
\curveto(294.962882,74.94860682)(325.506832,37.89705282)(325.506832,37.89705282)
\lineto(390.624022,67.19391282)
\lineto(455.721672,37.89705282)
\curveto(488.610162,74.94860682)(520.819332,37.89705282)(520.819332,37.89705282)
\lineto(584.958012,19.34332282)
}
}
{
\newrgbcolor{curcolor}{0 0 0}
\pscustom[linestyle=none,fillstyle=solid,fillcolor=curcolor]
{
\newpath
\moveto(269.12020639,8.74467642)
\lineto(269.12020639,2.62521656)
\lineto(267.45455459,2.62521656)
\lineto(267.45455459,8.69036169)
\curveto(267.45454608,9.64991494)(267.26746219,10.36807569)(266.89330235,10.84484608)
\curveto(266.51912662,11.32160004)(265.95787494,11.55998113)(265.20954564,11.55999006)
\curveto(264.3103297,11.55998113)(263.60122141,11.27332032)(263.08221862,10.7000068)
\curveto(262.56320755,10.12667712)(262.30370409,9.34514925)(262.30370745,8.35542084)
\lineto(262.30370745,2.62521656)
\lineto(260.6290032,2.62521656)
\lineto(260.6290032,12.76396663)
\lineto(262.30370745,12.76396663)
\lineto(262.30370745,11.18883939)
\curveto(262.70201173,11.79836221)(263.16972146,12.25400202)(263.70683804,12.55576016)
\curveto(264.24998006,12.85749844)(264.87459886,13.00837255)(265.58069631,13.00838293)
\curveto(266.74543778,13.00837255)(267.62654256,12.64627469)(268.2240133,11.92208828)
\curveto(268.82146548,11.20391824)(269.12019621,10.14478201)(269.12020639,8.74467642)
}
}
{
\newrgbcolor{curcolor}{0 0 0}
\pscustom[linestyle=none,fillstyle=solid,fillcolor=curcolor]
{
\newpath
\moveto(272.67782161,11.0440001)
\lineto(284.28306946,11.0440001)
\lineto(284.28306946,9.52318759)
\lineto(272.67782161,9.52318759)
\lineto(272.67782161,11.0440001)
\moveto(272.67782161,7.35059829)
\lineto(284.28306946,7.35059829)
\lineto(284.28306946,5.81168087)
\lineto(272.67782161,5.81168087)
\lineto(272.67782161,7.35059829)
}
}
{
\newrgbcolor{curcolor}{0 0 0}
\pscustom[linestyle=none,fillstyle=solid,fillcolor=curcolor]
{
\newpath
\moveto(289.80506751,4.16413398)
\lineto(296.18704858,4.16413398)
\lineto(296.18704858,2.62521656)
\lineto(287.60532084,2.62521656)
\lineto(287.60532084,4.16413398)
\curveto(288.29934037,4.88229319)(289.24381228,5.84486998)(290.43873939,7.05186726)
\curveto(291.63969308,8.26489065)(292.39406361,9.04641852)(292.70185324,9.39645321)
\curveto(293.28723832,10.05425755)(293.69459841,10.60947426)(293.92393473,11.06210501)
\curveto(294.15929066,11.52075386)(294.27697246,11.9703587)(294.27698049,12.41092087)
\curveto(294.27697246,13.12907183)(294.02350396,13.71446336)(293.51657423,14.16709722)
\curveto(293.01566493,14.619708)(292.36087131,14.84601916)(291.55219141,14.84603138)
\curveto(290.9788645,14.84601916)(290.37235059,14.74644225)(289.73264787,14.54730035)
\curveto(289.09897313,14.34813461)(288.42003966,14.0463864)(287.69584539,13.64205481)
\lineto(287.69584539,15.48875572)
\curveto(288.43210958,15.7844561)(289.12009551,16.00774978)(289.75980523,16.15863742)
\curveto(290.39950793,16.30949799)(290.98489946,16.38493504)(291.51598158,16.3849488)
\curveto(292.91608802,16.38493504)(294.03255641,16.03490712)(294.86539009,15.33486397)
\curveto(295.69820654,14.63479541)(296.11461907,13.69937595)(296.11462894,12.52860279)
\curveto(296.11461907,11.97337618)(296.0090072,11.44531681)(295.797793,10.94442309)
\curveto(295.59259467,10.4495477)(295.2154094,9.86415617)(294.66623607,9.18824674)
\curveto(294.51535355,9.01322621)(294.03557389,8.50628922)(293.22689566,7.66743423)
\curveto(292.41820347,6.83460412)(291.27759523,5.66683854)(289.80506751,4.16413398)
}
}
{
\newrgbcolor{curcolor}{0 0 0}
\pscustom[linestyle=none,fillstyle=solid,fillcolor=curcolor]
{
\newpath
\moveto(300.18823221,4.92454023)
\lineto(302.09830031,4.92454023)
\lineto(302.09830031,3.3675179)
\lineto(300.61369762,0.47073217)
\lineto(299.44593087,0.47073217)
\lineto(300.18823221,3.3675179)
\lineto(300.18823221,4.92454023)
}
}
{
\newrgbcolor{curcolor}{0 0 0}
\pscustom[linestyle=none,fillstyle=solid,fillcolor=curcolor]
{
\newpath
\moveto(305.66496941,12.76396663)
\lineto(307.33062121,12.76396663)
\lineto(307.33062121,2.44416745)
\curveto(307.3306178,1.15268528)(307.08318427,0.21726582)(306.58831987,-0.36209373)
\curveto(306.09948509,-0.94144731)(305.30890478,-1.2311256)(304.21657655,-1.23112945)
\lineto(303.58290467,-1.23112945)
\lineto(303.58290467,0.18105359)
\lineto(304.02647498,0.18105359)
\curveto(304.66014612,0.18105604)(305.09164606,0.32891266)(305.32097611,0.62462391)
\curveto(305.55030335,0.91430419)(305.66496767,1.5208181)(305.66496941,2.44416745)
\lineto(305.66496941,12.76396663)
\moveto(305.66496941,16.7108372)
\lineto(307.33062121,16.7108372)
\lineto(307.33062121,14.60161508)
\lineto(305.66496941,14.60161508)
\lineto(305.66496941,16.7108372)
}
}
{
\newrgbcolor{curcolor}{0 0 0}
\pscustom[linestyle=none,fillstyle=solid,fillcolor=curcolor]
{
\newpath
\moveto(311.02402446,11.0440001)
\lineto(322.62927231,11.0440001)
\lineto(322.62927231,9.52318759)
\lineto(311.02402446,9.52318759)
\lineto(311.02402446,11.0440001)
\moveto(311.02402446,7.35059829)
\lineto(322.62927231,7.35059829)
\lineto(322.62927231,5.81168087)
\lineto(311.02402446,5.81168087)
\lineto(311.02402446,7.35059829)
}
}
{
\newrgbcolor{curcolor}{0 0 0}
\pscustom[linestyle=none,fillstyle=solid,fillcolor=curcolor]
{
\newpath
\moveto(332.11624393,9.91244317)
\curveto(332.99130622,9.725352)(333.67325718,9.3360968)(334.16209885,8.74467642)
\curveto(334.65695635,8.15324381)(334.90438989,7.42301313)(334.9044002,6.55398221)
\curveto(334.90438989,5.22025118)(334.44573261,4.1882723)(333.52842697,3.45804246)
\curveto(332.61110347,2.72781095)(331.3075512,2.36269561)(329.61776623,2.36269535)
\curveto(329.05047457,2.36269561)(328.46508304,2.42002777)(327.86158988,2.534692)
\curveto(327.26412515,2.64332145)(326.64554131,2.80928297)(326.00583652,3.03257705)
\lineto(326.00583652,4.79780586)
\curveto(326.5127721,4.50209044)(327.06798881,4.27879676)(327.67148831,4.12792416)
\curveto(328.27498166,3.97704855)(328.90563543,3.9016115)(329.5634515,3.90161277)
\curveto(330.71008974,3.9016115)(331.58214207,4.12792265)(332.17961112,4.58054693)
\curveto(332.78309995,5.03316729)(333.08484817,5.69097839)(333.08485666,6.55398221)
\curveto(333.08484817,7.35059356)(332.80422233,7.97219488)(332.2429783,8.41878803)
\curveto(331.68775394,8.87140455)(330.91226104,9.09771571)(329.91649726,9.09772219)
\lineto(328.34137002,9.09772219)
\lineto(328.34137002,10.60042979)
\lineto(329.9889169,10.60042979)
\curveto(330.88812118,10.60042181)(331.57610711,10.77845326)(332.05287674,11.13452466)
\curveto(332.52963146,11.496614)(332.76801254,12.01562093)(332.76802072,12.69154699)
\curveto(332.76801254,13.38555781)(332.52057901,13.91663467)(332.02571937,14.28477914)
\curveto(331.53687984,14.65893527)(330.8338065,14.84601916)(329.91649726,14.84603138)
\curveto(329.4155899,14.84601916)(328.87847809,14.79170448)(328.30516019,14.68308718)
\curveto(327.73183488,14.57444577)(327.10118111,14.40546677)(326.41319701,14.17614968)
\lineto(326.41319701,15.80559166)
\curveto(327.10721608,15.99869733)(327.75597474,16.14353647)(328.35947493,16.24010952)
\curveto(328.96900255,16.33665533)(329.54232415,16.38493504)(330.07944146,16.3849488)
\curveto(331.46747775,16.38493504)(332.56584124,16.06809942)(333.37453523,15.43444098)
\curveto(334.18321166,14.80679189)(334.58755427,13.95586193)(334.58756426,12.88164855)
\curveto(334.58755427,12.13330273)(334.37331303,11.49963148)(333.94483992,10.98063291)
\curveto(333.51634811,10.4676526)(332.90681672,10.11158971)(332.11624393,9.91244317)
}
}
\end{pspicture}
\caption{Constructions of a more general case, $j=3$.}\label{constructionpic2}
\end{figure}

The first visualization of the projective limit construction of approximating quantum graphs is the account that Barlow and Evans give in \cite{BarlowEvans2004}. In this account the first graph is a unit interval, call it $F_0$. To construct $F_n$ from $F_{n-1}$ take two copies of $F_{n-1}$ (if the dimension is larger than two more copies are needed, see \cite{Steinhurst2009} for details of how many). With those two copies, identify pairs of points, i.e. wormholes, from both spaces to create a single space this will be $F_n$. It is important that the locations of the wormholes are chosen so that wormholes used for different $n$ do not occur at the same horizontal coordinate. In Figures \ref{repfig}\ and \ref{constructionpic2}\ one can see the first few steps for when $j=2$ and $j=3$. This visualization is very convenient for drawing pictures but not for the numerical approximations needed to describe a Laplacian and then to calculate its spectrum.

Another way to visualize the construction of the approximating graphs to the Laakso spaces is to think in terms of their cell structure. In general a Laakso space's cell structure is not self-similar but in the cases that we consider, where at each step we split every horizontal interval into the same number, $j$, of subintervals, the Laakso spaces do have a self-similar cell structure. Again we begin with $F_0$ being the unit interval. Now take $2j$ copies of $F_{n-1}$ arrange them in two horizontal lines of $j$ copies each. Figure \ref{constructionpic}\ shows an example of how this is done. To start with we have two copies of $F_0$ identified together at their mid-points, this is shown by the dots, to form $F_1$. To construct $F_2$ we take four copies of $F_1$ with the endpoints label so as to glue the graphs together as one would expect from the picture, but the extra step is to say that in the last part of the Figure that the points labeled ``1'' are identified together and the points labeled ``2'' are identified together. The process would them be continued to generate $F_3$. This example is when $j=2$ but in other cases one just has more instances of this process. This may seem like an inelegant mode of thought but it leads to simpler labeling schemes for the nodes when building the incidence matrices.

\begin{figure}[t]\center
\input{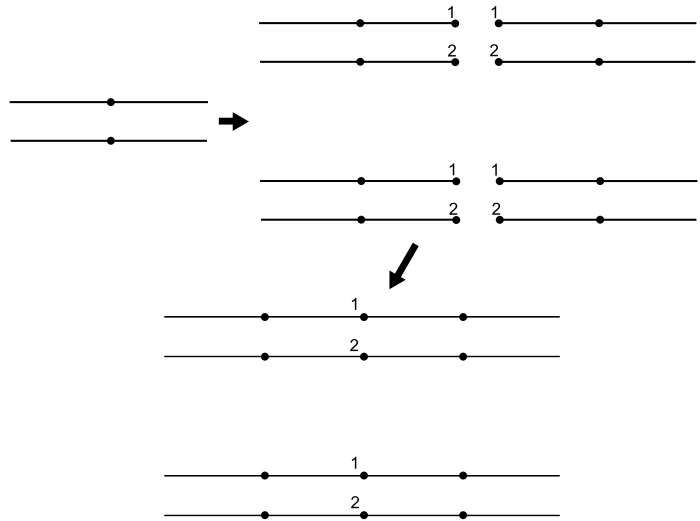}
\caption{Method of constructing the space when $j=2$}\label{constructionpic}
\end{figure}

To describe each of the approximating graphs for the numerical computations we use incidence matrices. The incidence matrix of a graph is a simple representation of how the nodes in space are connected. An entry of 0 indicates no connection while a 1 in row 1, column 2, would represent a connection between node 1 and node 2. This would also require node 2 to be connected to node 1, as the incidence matrix must be symmetric. An entry of 2 means that there are two separate edges connecting those two vertices. Many different incidence matrices are possible since the labeling of nodes is arbitrary. We assume a labeling natural from the second visualization of the construction process in the previous paragraph. The incidence matrix according to the vertex labeling scheme we use for the $n=1$, $j=2$ space is
$$\left[ \begin{array}{rrrrr}
0&0&1&0&0\\
0&0&1&0&0\\
1&1&0&1&1\\
0&0&1&0&0\\
0&0&1&0&0 \end{array} \right],$$
The incidence matrix for the $n=1$, $j=3$ case is
$$\left[ \begin{array}{rrrrrr}
0&0&1&0&0&0\\
0&0&1&0&0&0\\
1&1&0&2&0&0\\
0&0&2&0&1&1\\
0&0&0&1&0&0\\
0&0&0&1&0&0
\end{array} \right].$$
The numbering scheme used for these matrices started in the top left node and went down columns of nodes, left to right across the space. It is the same scheme as shown in Figures~\ref{repfig} and \ref{constructionpic2}. These matrices are needed to create the Laplacian matrices in the next section.

From \cite{Laakso2000} we know that our space has a Hausdorff dimension $Q$ given by $$1+\ln(2)/\ln(j)=Q.$$ Since $j$ is an integer larger than one the set of dimensions we deal with are contained in $[1,2]$. The more general construction given in \cite{Steinhurst2009} allows for different values of $j$ at each step. This is necessary to get all possible dimensions possible according to \cite{Laakso2000}. In these cases we deal with the sequence $\{j_i\}_{i=0}^{\infty}$ listing the choice of $j$ at the $i$'th step in the construction.

\section{Laplacian}\label{sec:lap}
\begin{table}[t]\center\small
$\begin{array}{cc|cc|cc|cc|cc|cc|l}
\multicolumn{2}{c|}{j=2} & \multicolumn{2}{|c|}{j=3} & \multicolumn{2}{|c|}{j=4} & \multicolumn{2}{|c|}{j=5} & \multicolumn{2}{|c|}{j=6} & \multicolumn{2}{|c|}{j=7} & Expected\\
\lambda & m&\lambda & m&\lambda & m&\lambda & m&\lambda & m&\lambda & m& \lambda\\ \hline
0&1&0&1&0&1&0&1&0&1&0&1&0\\
9.87&3&9.87&1&9.87&1&9.87&1&9.87&1&9.87&1&\pi^{2} \\
&&22.21&2&&&&&&&&&(1.5\pi)^{2}\\
39.48&3&39.48&1&39.48&3&39.48&1&39.48&1&39.48&1&(2\pi)^{2}\\
&&&&&&61.68&2&&&&&(2.5\pi)^{2}\\
88.82&3&88.83&2&88.83&1&88.83&1&88.83&3&88.83&1&(3\pi)^{2}\\
&&&&&&&&&&120.90&2&(3.5\pi)^{2}\\
157.88&18&157.91&1&157.91&3&157.91&1&157.91&1&157.91&1&(4\pi)^{2}\\
&&199.86&8& &&&&&&&&(4.5\pi)^{2}\\
246.66&3&246.74&1&246.74&1&246.74&4&246.74&1&246.74&1&(5\pi)^{2}\\
355.15&6&355.30&2&35.31&3&355.30&1&355.31&5&355.30&1&(6\pi)^{2}\\
483.31&3&483.61&1&483.61&1&483.61&1&483.61&1&483.61&6&(7\pi)^{2}\\
&&&&&&555.16&2&&&&&(7.5\pi)^{2}\\
631.15&66&&&621.65&10&&&631.65&1&631.65&1&(8\pi)^{2}\\
798.63&3&&&799.43&1&&&799.44&3&&&(9\pi)^{2}
\end{array}$
\caption{Calculated Values of the first 10 Eigenvalues for $j=2,3,4,5,6,7$ with multiplicity, $m$, and the Expected value, $\lambda$.}\label{tab:eigs}
\end{table}

In this section we describe how we numerically approximate the Laplacian in Definition \ref{def:A_n}. On each approximating graph an operator acts on a finite dimensional space of functions and so can be given a matrix representation. The matrix of the Laplacian is generated by simple manipulations of the incidence matrix and is constructed to follow the second difference quotient formula, 
$$u^{''}(x)=\lim_{h \to 0}\frac{u(x+h)-2u(x)+u(x-h)}{h^2}.$$ 
However, as most nodes are connected to two nodes on each side, we took the average of the second derivative along the two paths. To construct the Laplacian's matrix representation from the incidence matrix, scale each row so that the row sum is $-2$ and then set all diagonal entries to $2$. Now the row sum is always zero.  If one considers functions depending only on the horizontal coordinate, i.e. a function on the unit interval, this Laplacian is coincides with the negative second derivative on the unit interval. This representation also  imposes  Neumann boundary conditions at the endpoints by reflecting the value at the points just inside the boundary of the space to the points just outside. This makes the ``first derivative'' zero on all boundary points of the space and allows the second derivative to be approximated there. The entire matrix is then scaled by $1/h^2$, where $h$ is the distance between nodes or wormholes in the space. This distance decreases with increasing $n$ such that the total width of the graph is 1 and $h$ is given by $h(n)=1/j^n.$ The algorithmic description is desirable here to indicate how the calculations were actually performed.

The Laplacian matrix for the $n=1$, $j=2$ space is
$$M_{1,2} = \left[ \begin{array}{rrrrr}
8&0&-8&0&0\\
0&8&-8&0&0\\
-2&-2&8&-2&-2\\
0&0&-8&8&0\\
0&0&-8&0&8 \end{array} \right].$$
The Laplacian matrix for the $n=1$, $j=3$ space is
$$M_{1,3} = \left[ \begin{array}{rrrrrr}
18&0&-18&0&0&0\\
0&18&-18&0&0&0\\
-4.5&-4.5&18&-9&0&0\\
0&0&-9&18&-4.5&-4.5\\
0&0&0&-18&18&0\\
0&0&0&-18&0&18
\end{array} \right].$$
The notation is $M_{n,j}$ is the Laplacian matrix for the $n'th$ level approximating graph to the Laakso space constructing using the value of $j$.

\begin{figure}[t]\center
\subfigure[$j=2$]{\includegraphics[scale=0.4]{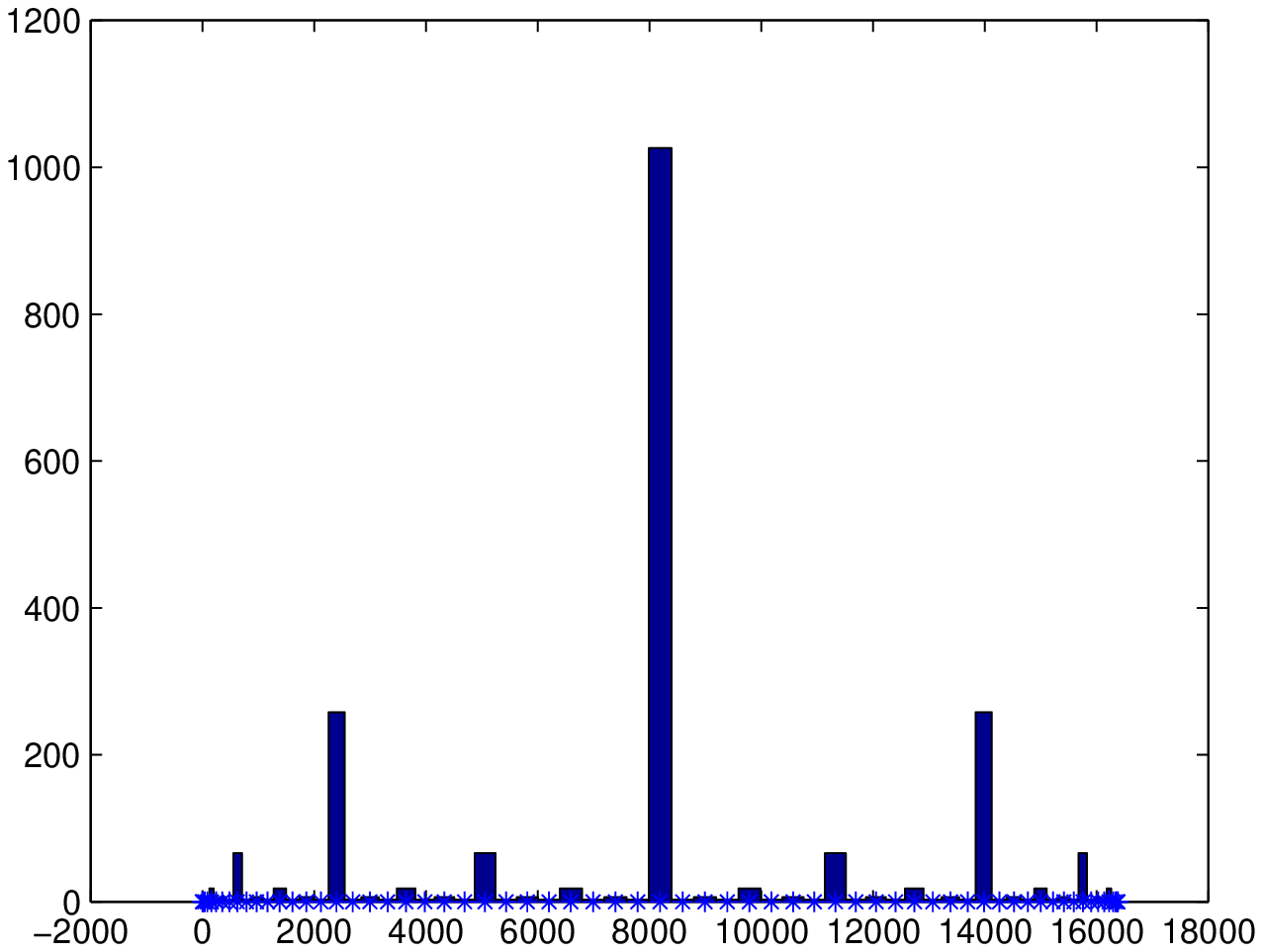}}
\subfigure[$j=3$]{\includegraphics[scale=0.4]{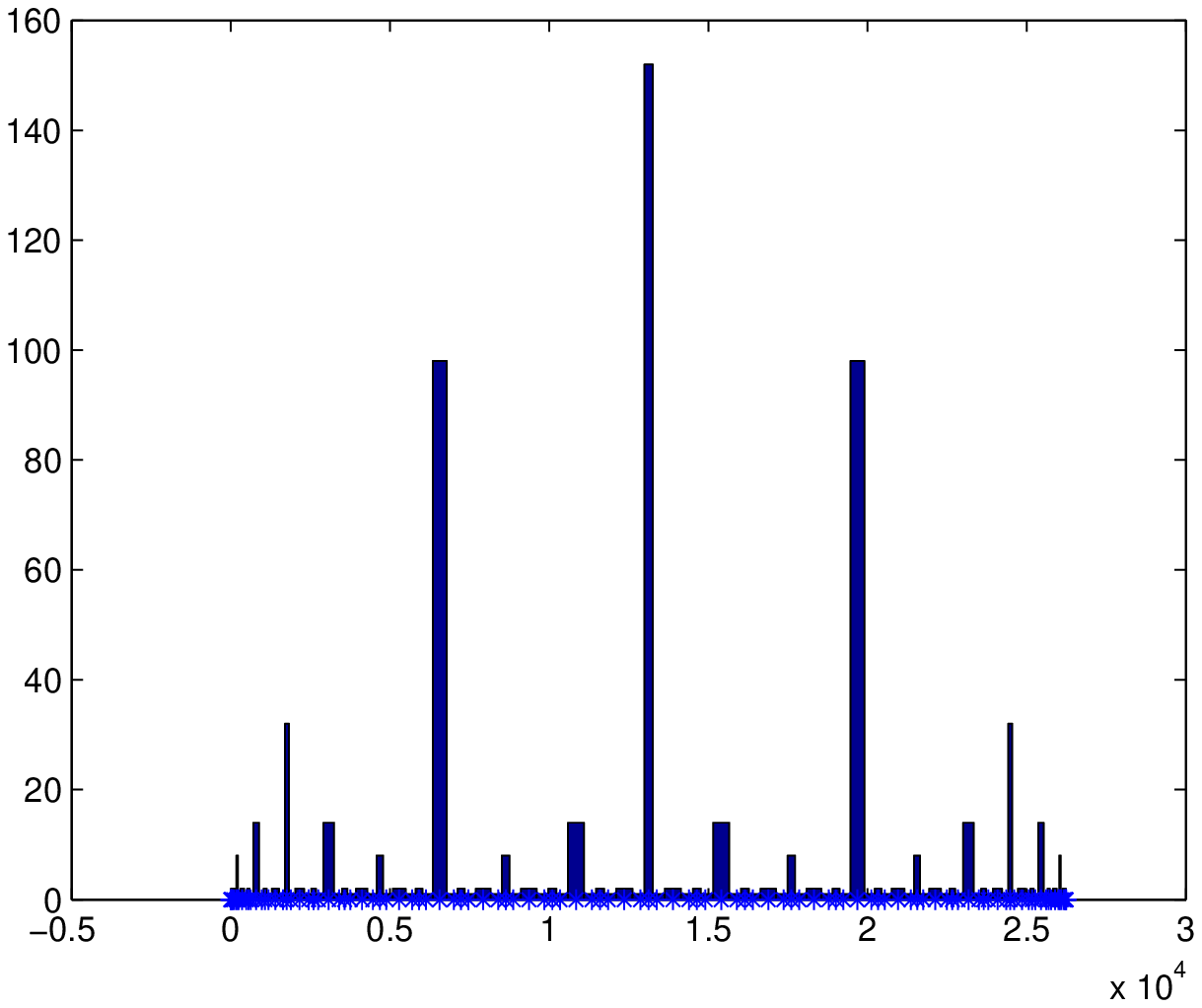}}
\subfigure[$j=4$]{\includegraphics[scale=0.4]{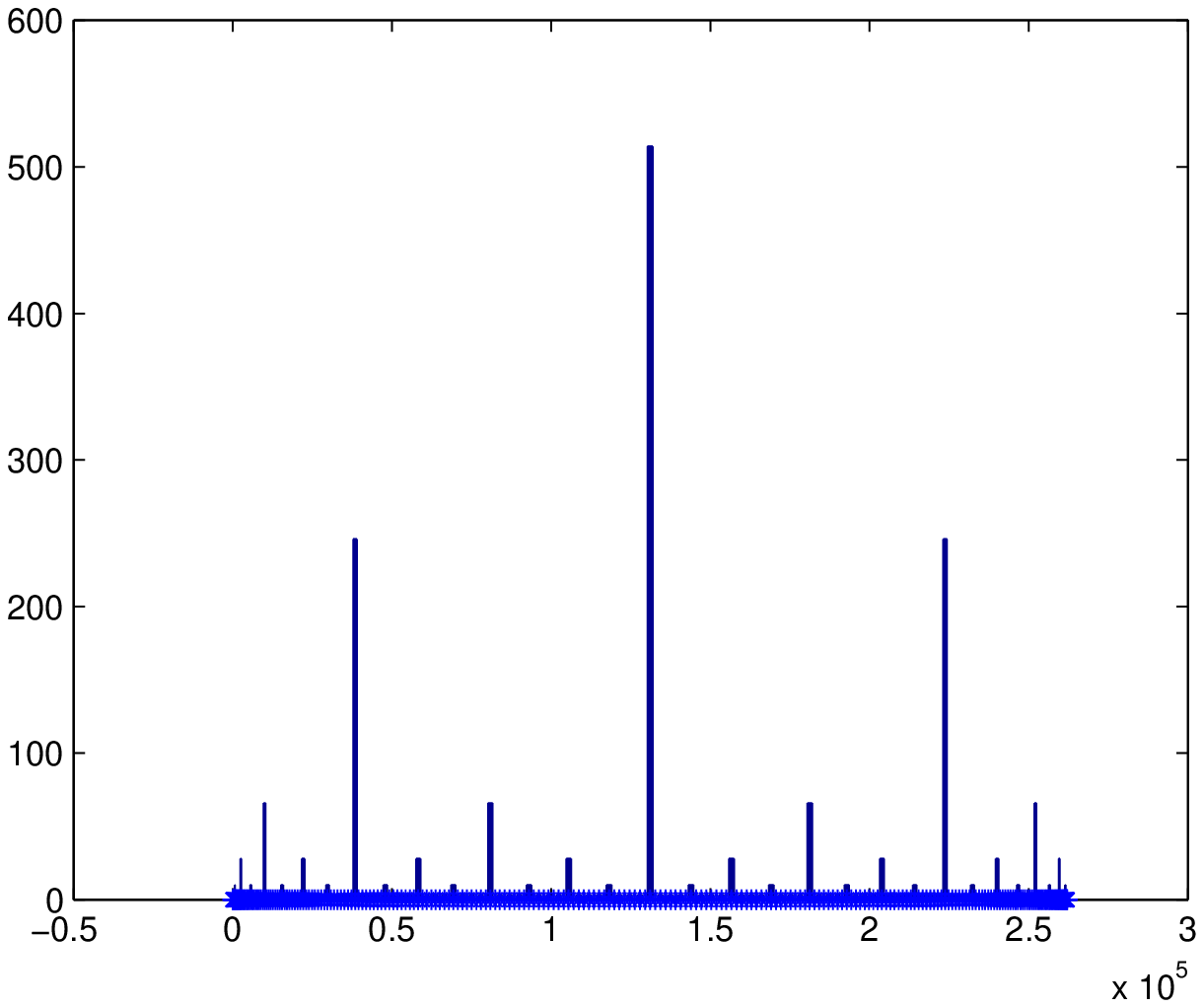}}
\subfigure[$j=5$]{\includegraphics[scale=0.4]{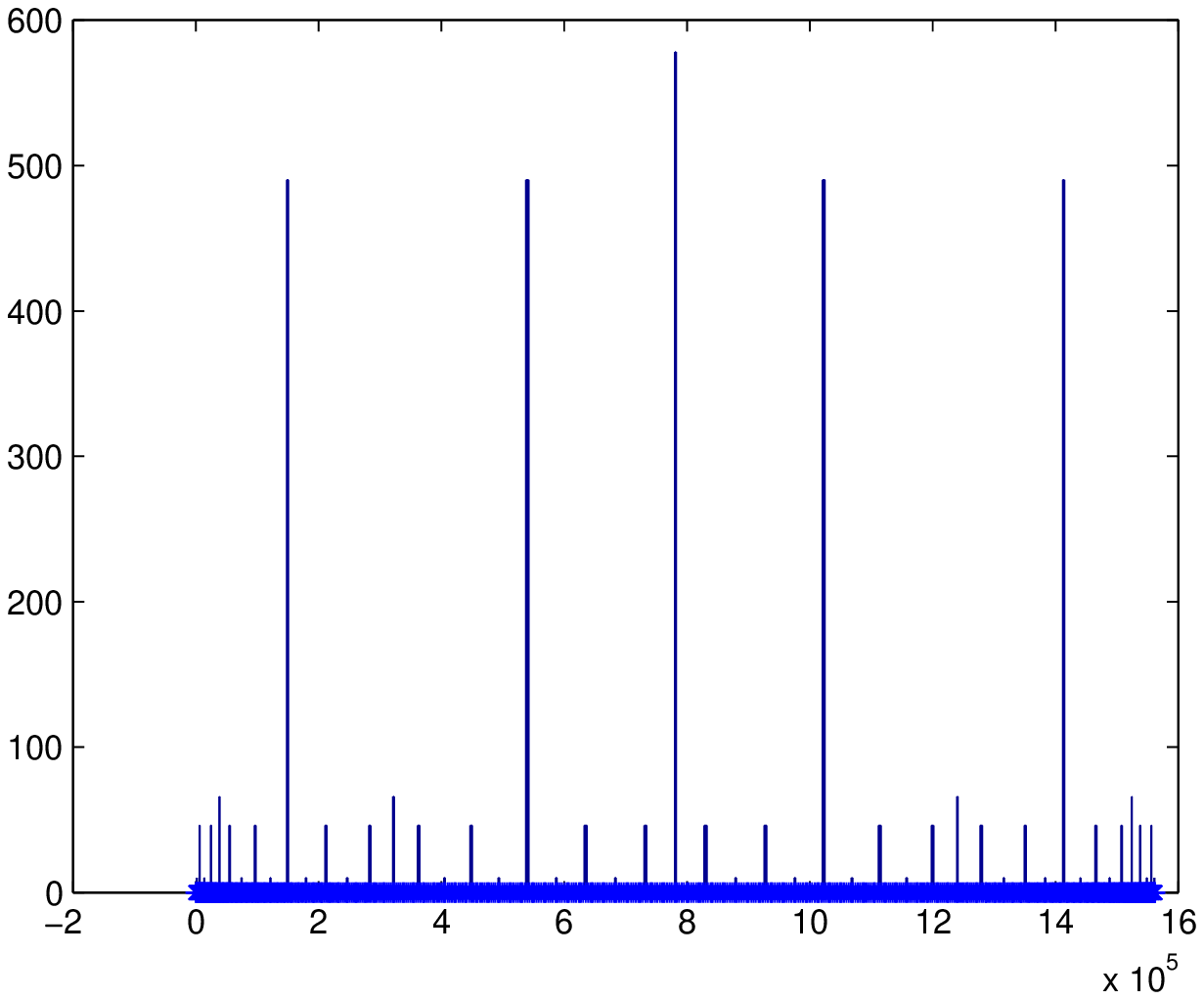}}
\caption{Multiplicities of eigenvalues when $j$ is $2,3,4$, and $5$ respectively, $n=4$ except for the (a) where $n=6$}\label{fig:hist}
\end{figure}

This algorithm was implemented in MATLAB. It first produces the incidence matrix from recursive patterns emerging from the labeling scheme for the nodes. Then the Laplacian matrix is constructed using the process outlined above. Since these matrices are well structured and very sparse, no row has more than five entries compared to thousands of zeroes, we were able to find the thousand smallest eigenvalues numerically using the MATLAB \emph{eigs} command which is based on ARPACK, a package that implements an Implicitly Restarted Arnoldi Method. More information on ARPACK can be found in its users guide, \cite{Arpack1998}. Even using sparse storage techniques the computations are impractical on a personal computer much past $n=9$ when $j=2$.

\begin{figure}[t]\center
\subfigure[$\lambda=9.74$]{\includegraphics[scale=0.4]{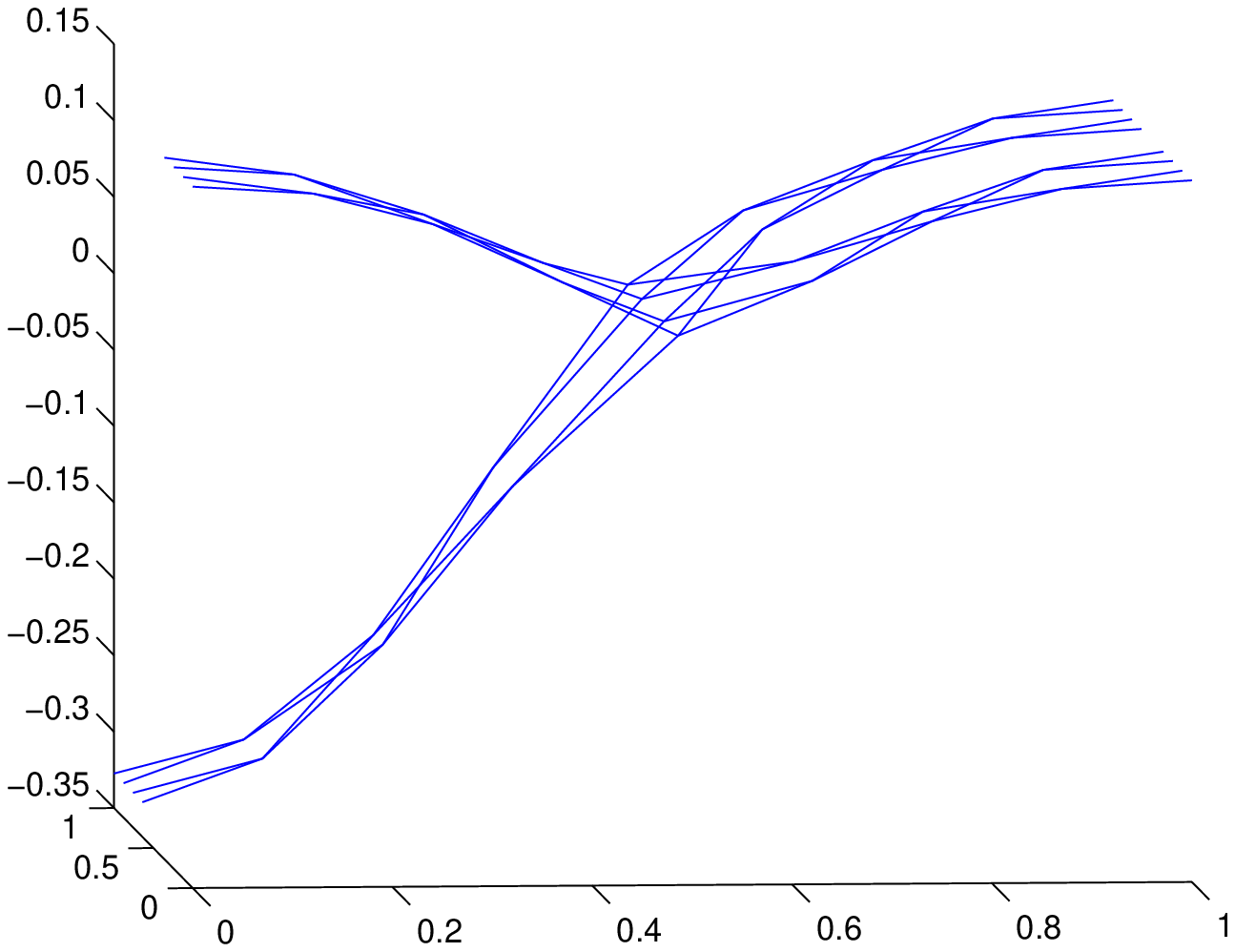}}
\subfigure[$\lambda=37.49$]{\includegraphics[scale=0.4]{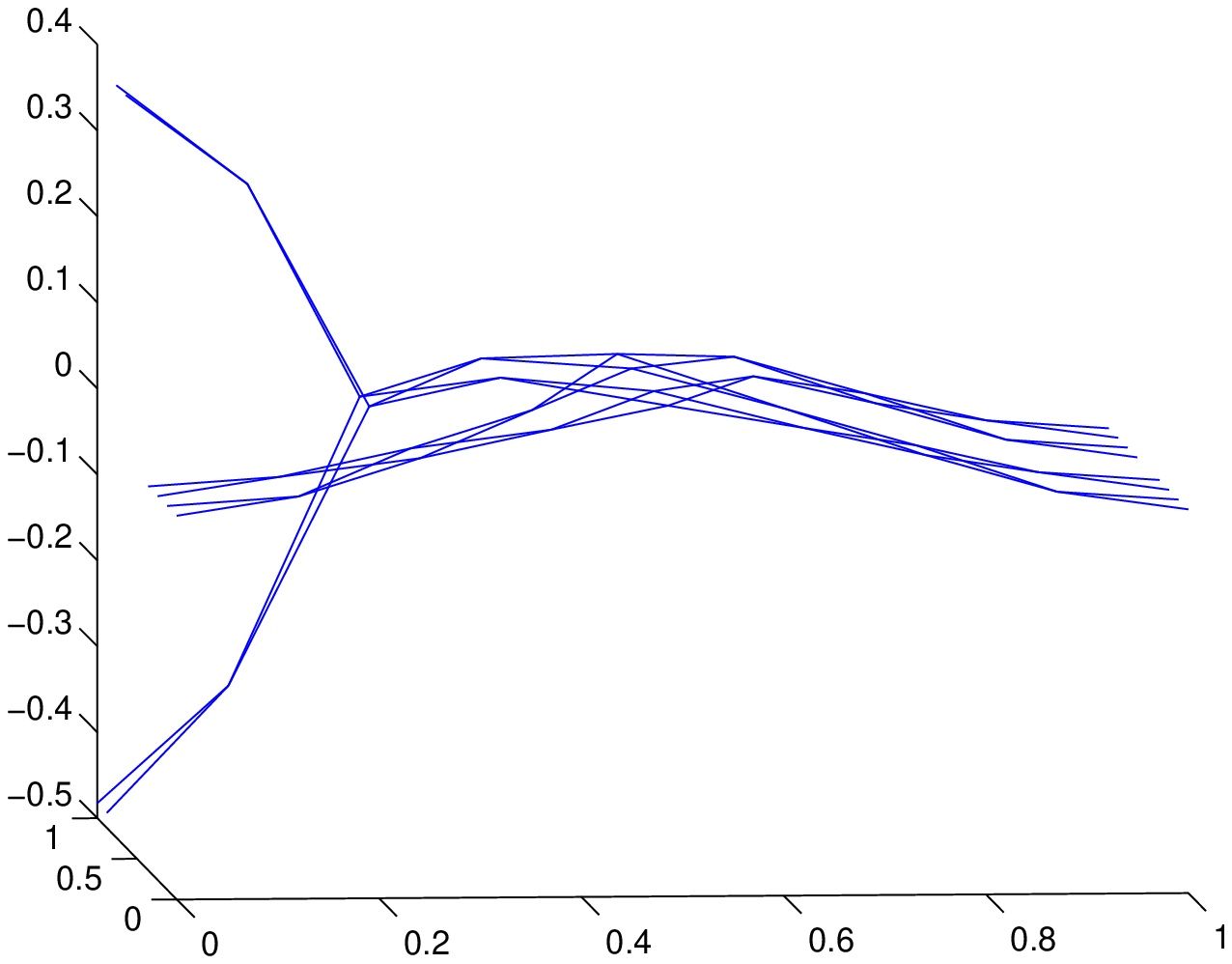}}
\caption{Eigenfunctions on the $n=3$, $j=2$ space}\label{eigf s2}
\end{figure}

Fix the Laakso space, $L$, where $j$ is the number of subintervals in the construction. The action of the operator $M_n \phi^{*}_{n,m}, n >m,$ on $f \in C(F_m)$ needs to be related to the action of $A_m$ to justify approximating $A$ by $M_n$ instead of $A_n$. First $\phi_{n,m}:F_n \rightarrow F_m$ is a surjection, so $\phi^{*}_{n,m}:C(F_m) \rightarrow C(F_n)$ by $\phi^{*}_{n,m}f = f \circ \phi_{n,m}$. Since $M_n$ is a second finite difference operator on the graph of $F_n$ it is seen that $M_n \phi^{*}_{n,m}$ for $n > m$ are the standard second finite difference operators on each line segment of $F_m$ which converge to second differentiation as $n$ goes to infinity. Then this holds for functions on any $F_m$ so the $M_n$ can be taken as suitable approximations to second differentiation. The operator $A_n$ (Definition \ref{def:A_n}) acts by second differentiation on each line segment of the quantum graph $F_m$. Then for fixed $j$ and $m$ and $f\in Dom(A_m)$
$$\lim_{n \rightarrow \infty} M_n \phi^{*}_{n,m} f = A_m f.$$
The proof of this statement is standard and omitted, but is essentially just a statement about approximating derivatives finite difference operators.

Using the matrix representations of $M_n$ given earlier in this section numerical estimates of the eigenfunctions and eigenvalues were computed using MATLAB's \emph{eigs} function which calls the ARPACK package, \cite{Arpack1998}. The ten lowest eigenvalues for each of the spaces with $j = 2,3,4,5,6,7$ are listed in Table \ref{tab:eigs} along with their observed multiplicities, and histograms given in Figure \ref{fig:hist}. In the next section we give a theoretical account of how these spectrums with multiplicities arise and the calculations to show the exact details for $j=2,3$ as the last result of this paper. Figure \ref{eigf s2} show two eigenfunctions labeled with eigenvalues for the Laakso space with $j=2$. These two pictures show how high multiplicity can be obtained since the eigenfunctions are only obeying the Kirchoff matching condition at the vertices of the graph. These matching conditions are weak enough so that functions which do not appear smooth at the vertices are still in the domain of the Laplacian.

\section{Spectrum of The Laplacian}\label{sec:spectrum}
In this section we use the framework from Theorem \ref{thm:orthodom} to verify that the calculations in the previous section are valid. Fix a Laakso space $L_j = L$, and the sequence of quantum graphs of which it is the projective limit, $\{F_n\}_{n=0}^{\infty}$, \cite{BarlowEvans2004,Steinhurst2009}. Let $\Phi_n:L \rightarrow F_n$ where $L$ is a Laakso space and $F_n$ one of the approximating quantum graphs be a projection map. It is possible to pull back function on $F_n$ to functions on $L$ by pre-composing with $\Phi_n$. I.e. if $f:F_n \rightarrow \mathbb{R}$ then $\Phi_n^{*}f := f \circ \Phi_n : L \rightarrow \mathbb{R}$. We can then take $A = A_j$ to be the self-adjoint operator with it's domain $Dom(A)$ where this operator is the projective limit of $A_n$ which acts by second differentiation on each line segment of $F_n$ with domain $Dom(A_n)$ as given in Definition \ref{def:A_n}. Recall also the orthogonalization of the domains of $A_n$ as $\C{D}'_n$ from Theorem \ref{thm:orthodom}.
 	
In the case of the Laakso spaces it will be useful to have a clear picture of what the projection maps $\phi_n$ do to a function from $C(F_{n-1})$ as it is being pulled back to an element of $C(F_n)$ to better understand what the orthogonal space is. Since $F_n$ is constructed from copies of $F_{n-1}$ modulo identifications, we can think of $F_n = F_{n-1} \times G_n$ so a function from $C(F_{n-1})$ when pulled back to $C(F_n)$ will be constant across $G_n$ for a given coordinate in $F_{n-1}$. However, at wormholes $G_n$ is identified to a single point so the pull back of a $C(F_{n-1})$ function will just have its value there. If we want to describe $\mathcal{D'}_n = \mathcal{D'}_{n-1}^{\perp} \cap \C{D}_n.$ we have to describe the functions on $F_n$ that are orthogonal to those on $F_{n-1}$. The functions pulled back from $C(F_{n-1})$ to $C(F_n)$ are constant across $G_n$ so the orthogonal functions would be those that average to zero across $G_n$ and at the $n-th$ level wormholes where $G_n$ is identified to a single point the orthogonal functions must be equal to zero.
 
In \cite{Laakso2000} a sequence of integers $\{j_i\}$ is associated to each Laakso space in the examples we calculated we have taken $j=j_i$ but we use the Laakso's notation for the generality since nowhere in the following proof do we use the fact that the $j_i$ are equal. Let $A=A_j$ be the self-adjoint operator for the Laakso space with $j_i=j$, and $A_n = A_{n,j}$ the self-adjoint operators on the approximating quantum graphs.

\begin{theorem}\label{thm:laaksospec}
Let the sequence $d_n = \prod_{j=1}^{n} j_i^{-1}$ $n \ge 0$ be associated to a given Laakso space, with Hausdorff dimension, $Q$, between one and two. Set $d_0 = 1$. Then the spectrum of $A$ on this Laakso space is 
$$\sigma(A) = \bigcup_{n=0}^{\infty} \bigcup_{k=1}^{\infty} \left\{ \frac{k^{2}\pi^{2}}{d_n^{2}} \right\} 
\cup
 \bigcup_{n=2}^{\infty} \bigcup_{k=1}^{\infty} \left\{ \frac{k^{2}\pi^{2}}{4d_n^{2}} \right\} 
\cup 
\bigcup_{n=1}^{\infty} \bigcup_{k=0}^{\infty} \left\{ \frac{(2k+1)^{2}\pi^{2}}{4d_n^{2}} \right\}.$$ 
\end{theorem}

\begin{proof}
The proof falls naturally into two parts. The first is to show that $A_n$ restricted to $\mathcal{D'}_n$ has the claimed spectrum. The second is to show that the union of $\sigma(A_n)$ contains all of $\sigma(A)$. Calculating exact multiplicities is left for the next theorem. 

The approximating graphs for Laakso spaces are introduced in \cite{Steinhurst2009} and have been reviewed Section \ref{sec:desc}. For example if $j_i = 3$ then we have the graph in Figure \ref{constructionpic2} as the second approximating quantum graph. The structure that we exploit to calculate the spectrum is that the graph can be broken down into simple configurations of line segments on which the behavior of $A_{n,j}$ is the usual second differentiation. All of the operators, $A_n$, have Kirchoff matching conditions which imposes Neumann boundary conditions since the boundary consists of degree one vertices. 

The functions in $\mathcal{D'}_n$ are those functions in the domain of the operator $A$ that are expressible as $\Phi_n^{*}f$ where $f \in Dom(A_n)$ where the values of $f$ summed over the copies of $F_{n-1}$ that $F_n$ is built from are all zero, especially at the $n-th$ level wormholes where all the copies of $F_{n-1}$ have been collapsed to single points $f$ must be equal to zero at these wormholes. Referring back to Figure \ref{constructionpic2}, if one considers the two most upper left line segments forming the sideways ``V''. An element of $\mathcal{D'}_2$ would have to equal zero at the second level wormhole, i.e. the apex of the V. It must also sum to zero over the copies of $F_1$ from which this graph is made meaning that the value of $f$ on the upper branch of the V must be opposite to the value of $f$ on the lower branch. On this ``V,'' the part of the function that is symmetric across the two branches is in $\mathcal{D'}_1$, while the anti-symmetric part is in $\mathcal{D'}_2 \subset \C{D}_2$.

The zeros at the $n'th$ level wormholes break the graph $F_n$ into three kinds of pieces. The first are those segments at the boundary which have length $d_n$, Neumann boundary conditions on the boundary end and Dirichlet boundary conditions on the interior end. If $j_n> 2$ then we have interior intervals, such as those that form the loops in Figure \ref{constructionpic2}, which again have length $d_n$ butthese have Dirichlet boundary conditions at each end. The third configuration is a cross, four segments of length $d_n$ joined at a common point and Dirichlet boundary conditions at the four outer vertices. These three configurations can be seen in Figure \ref{fig:pieces}. In the cross we must note that the upper and lower branches are from different copies of $F_{n-1}$ so $f$ is not forced to equal $-g$ for functions in $\C{D'}_n$ but we will consider the symmetric and anti-symmetric cases where $f = g$ and $f = -g$ from which any function on the cross is a combination. If $f=g$ then the cross is essentially just an interval of length $2d_n$ with Dirichlet boundary condition and if $f = -g$ then at the intersection $f=0$ and we have two intervals of length $d_n$ with Dirichlet boundary conditions. 

\begin{figure}[t]
\begin{center}
	\includegraphics[scale=.6]{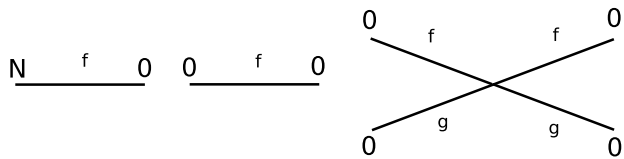}
\caption{The three types of pieces that the orthogonality condition creates in $F_n$.}
\label{fig:pieces}
\end{center}
\end{figure}

Note that the crosses do not appear for $n=0$ or $n=1$ for any $j_0$ or $j_1$. The intervals with Dirichlet conditions do not appear in $n=0$ or when $j_n = 2$. The intervals with mixed boundary conditions do not appear when $n=0$ either. What does appear when $n=0$ is an interval of unit length and both ends having Neumann boundary conditions. When all of these are put together and their spectra found by the usual methods we get the claimed spectra where the contributions towards multiplicity from each of the pieces may be zero occasionally, i.e. $j_i = 2$ so there are no intervals of length $d_n$ and Dirichlet boundaries except those which are part of the crosses. 

The last claim that these are all of the eigenvalues of $A$ on $L$ follows from the fact that $\bigcup_{n=0}^{\infty} \mathcal{D'}_n$ is dense in $Dom(A_j)$ by the projective limit construction of the operator $A$. So if there were another eigenvalue unaccounted for there would be at least a one dimensional subspace of $Dom(A)$ orthogonal to a dense subset of $Dom(A)$, since this can't happen the entire spectrum is now accounted for.
\end{proof}

There is a more general statement for higher dimension that replaces the crosses with $2^{k+1}$ line segments of length $d_n$ on each side of a central vertex where $k$ is the number of Cantor sets in the construction of the Laakso space as in \cite{Laakso2000}. The spectrum given in these higher dimensions is given by the same expression but the multiplicity counting as in the following Proposition has to be redone for each $k$. 

\begin{prop}
The first ten eigenvalues of $A_j$ and their multiplicity for the $j=2$ and $j=3$ constructions determined from Theorem \ref{thm:laaksospec}\ agree with the computed values in Table \ref{tab:eigs}.
\end{prop}

\begin{proof}
It is immediate from the spectrum of $A_{n,j}$ on a given Laakso space as determined in Theorem \ref{thm:laaksospec} that the lowest eigenvalue is increasing with $n$, so to find the first ten eigenvalues only a small number of approximating graphs are necessary.

Let us establish some notation. Let $\sigma^{N}[0,d]^{N}$ be the spectrum of the usual second derivative on an interval of length $d$. If either, or both, of the $N$'s are replaced by $0$ then Dirichlet boundary conditions are indicated instead. Denote by $\sigma_0^{0}[X,d]_0^{0}$ the spectrum of second differentiation on a cross where the four intervals have length $d$ and Dirichlet boundary conditions at the four ends points. If we write $2 \times \sigma_0^{0}[X,d]_0^{0}$ the two denotes the multiplicity of these eigenvalues. 

For $j=2$, we get from drawing the first few approximating graphs and forming the function spaces $\mathcal{D'}_n$ and calculating the spectrum of $A_{n,j}$ on each of these spaces we get the following table. Let $\sigma_n$ be the spectrum of the $A_{n,j}$ operator on the function space $\mathcal{D'}_n.$

\begin{eqnarray*}
	\sigma_0 &=& 1 \times  \sigma^{N}[0,1]^{N}\\
	\sigma_1 &=& 2 \times \sigma^{N}[0,\frac{1}{2}]^{0}\\
	\sigma_2 &=& 4 \times \sigma^{N}[0,\frac{1}{4}]^{0} + 1 \times \sigma_0^{0}[X,\frac{1}{4}]_0^{0}\\
	\sigma_3 &=& 8 \times \sigma^{N}[0,\frac{1}{8}]^{0} + 6 \times \sigma_0^{0}[X,\frac{1}{8}]_0^{0}\\
	\sigma_4 &=& 16 \times \sigma^{N}[0,\frac{1}{16}]^{0} + 28 \times \sigma_0^{0}[X,\frac{1}{16}]_0^{0}
\end{eqnarray*}

It is easy to check that $1 \times \sigma_0^{0}[X,d]_0^{0} = 1 \times^{0}[0,2d]^{0} + 2 \times \sigma^{0}[0,d]^{0}$. Using the known spectra for second differentiation on intervals the first ten eigenvalues are readily computed along with their multiplicities which do match those found in Table \ref{tab:eigs}. In the case where $j=3$ we need fewer approximating graphs to account for the first ten eigenvalues. The spectra of the Laplacian on the first few functions spaces $\mathcal{D'}_n$ are:

\begin{eqnarray*}
	\sigma_0 &=& 1 \times \sigma^{N}[0,1]^{N}\\
	\sigma_1 &=& 2 \times \sigma^{N}[0,\frac{1}{3}]^{0} + 1 \times \sigma^{0}[0,\frac{1}{3}]^{0}\\ 
	\sigma_2 &=& 4 \times \sigma{N}[0,\frac{1}{9}]^{0} + 3 \times \sigma^{0}[0,\frac{1}{9}]^{0} + 2 \times \sigma_0^{0}[X,\frac{1}{9}]_0^{0}
\end{eqnarray*}
When these spectra are written out with the indicated multiplicities and compiled into a single list we get the same eigenvalues and multiplicities listed in Table \ref{tab:eigs}.
\end{proof}

\end{document}